\documentclass[a4paper, 11pt, leqno]{amsart}

\usepackage{amsmath,amssymb,amsthm}
\usepackage{graphicx}
\usepackage{enumerate}

\usepackage{verbatim}
\usepackage{amssymb}
\usepackage{amsbsy}
\usepackage{amscd}
\usepackage{amsmath}
\usepackage{amsthm}
\usepackage{mathrsfs}

\newtheorem{theorem}{Theorem}\numberwithin{theorem}{section}

\newtheorem{lemma}[theorem]{Lemma}

\newtheorem{conj}[theorem]{Conjecture}

\theoremstyle{remark}
\newtheorem{remark}[theorem]{Remark}
\newtheorem{example}[theorem]{Example}

\numberwithin{equation}{section}

\def\AA{{A}}

\def\Tcal{\mathcal{T}^\sh}
\def\TT{T}

\def\Q{\mathbb{Q}}

\long\def\comment#1{}

\def\ZZ{\mathcal{Z}}

\DeclareFontEncoding{OT2}{}{}
\DeclareFontSubstitution{OT2}{cmr}{m}{n}
\DeclareSymbolFont{cyss}{OT2}{wncyss}{m}{n}
\DeclareMathSymbol{\sh}{\mathbin}{cyss}{`x}

\allowdisplaybreaks

\title{On a variant of multiple zeta values of level two}

\author[M. Kaneko]{Masanobu Kaneko}
\author[H. Tsumura]{Hirofumi Tsumura}

\address{M.\,Kaneko: Faculty of Mathematics, Kyushu University, Motooka 744, Nishi-ku, Fukuoka 819-0395, Japan}
\email{mkaneko@math.kyushu-u.ac.jp}

\address{{H.\,Tsumura:} Department of Mathematical Sciences, Tokyo Metropolitan University, 1-1, Minami-Ohsawa, Hachioji, Tokyo 192-0397, Japan}
\email{tsumura@tmu.ac.jp}



\begin{document}

\begin{abstract}
We study a variant of multiple zeta values of level 2, which forms a subspace of 
the space of alternating multiple zeta values. This variant, which is regarded as 
the `shuffle counterpart' of Hoffman's `odd variant', exhibits nice properties
such as duality, shuffle product, parity results, etc., like ordinary multiple zeta values.
We also give some conjectures on relations between our values, Hoffman's values,
and multiple zeta values. 
  
\end{abstract}

\maketitle

%\baselineskip 16pt
%%%%%%%%%%%%%%%%%%%%%%%%%%%%%%%%%%%%%%%%%%%%                                   
\section{Introduction} \label{sec-1}
%%%%%%%%%%%%%%%%%%%%%%%%%%%%%%%%%%%%%%%%%%%%  

In this paper, we study the following variant of the multiple zeta value,
\[ \sum_{0<m_1<\cdots <m_r\atop m_i\equiv i\bmod 2}
\frac1{m_1^{k_1}m_2^{k_2}\cdots m_r^{k_r}},\]
which was introduced in \cite[Section 5]{KT2} in connection with a `level 2' 
generalization of the zeta function studied by Arakawa and the first named
author in \cite{AK1999}.  We regard this value as a level 2 multiple zeta value
because of the congruence condition in the summation and of the (easily proved) fact that 
this value can be written as a linear combination of alternating multiple zeta values (also referred to as 
Euler sums or colored multiple zeta values)
\begin{equation}\label{eulersum} \sum_{0<m_1<\cdots <m_r}
\frac{(\pm1)^{m_1}(\pm1)^{m_2}\cdots(\pm1)^{m_r}}{m_1^{k_1}m_2^{k_2}\cdots m_r^{k_r}}.
\end{equation}
It turned out that the value with the normalizing factor $2^r$,
\begin{equation}\label{MTV}  
T(k_1,k_2,\ldots,k_r):=2^r\!\!\!\sum_{0<m_1<\cdots <m_r\atop m_i\equiv i\bmod 2}
\frac1{m_1^{k_1}m_2^{k_2}\cdots m_r^{k_r}},
\end{equation} was more natural and convenient, and we often refer this value as the `multiple $T$-value'  (MTV).

This is in contrast to Hoffman's multiple $t$-value defined by
\begin{equation}\label{hoft} t(k_1,\ldots,k_r)=\sum_{0<m_1<\cdots <m_r\atop \forall m_i\,:\text{ odd}}
\frac1{m_1^{k_1}m_2^{k_2}\cdots m_r^{k_r}},
\end{equation}
which was introduced and studied in his recent paper \cite{Hoffman16} as another variant of 
multiple zeta values of level 2. In the next section and in \S\ref{relTt}, we discuss in more detail 
the comparison between $T$- and $t$-values.

In the case of depth $r=2$, Tasaka and the first named author studied in \cite{KanekoTasaka}
both versions in connection to modular forms of level 2, and gave some results generalizing the previous 
work by Gangle-Kaneko-Zagier \cite{GKZ}. We do not pursue any modular aspects in this paper.

In the following sections, we show several properties of MTVs such as an integral expression, 
the duality relation, certain sum formulas, the parity result,  and the generating series of 
`height one' MTVs, all similar to those properties for classical multiple zeta values (MZVs)
\begin{equation}\label{mzv}
\zeta(k_1,\ldots,k_r)=\sum_{0<m_1<\cdots <m_r}\frac {1}{m_1^{k_1}m_2^{k_2}\cdots m_r^{k_r}},
\end{equation}
and give some conjectures concerning the space spanned by the multiple $T$-values
and a speculation on a basis of multiple zeta values in terms of Hoffman's $t$-values.

%%%%%%%%%%%%%%%%%%%
\section{The space of multiple $T$-values}\label{sec-2}
%%%%%%%%%%%%%%%%%%%

As in the study of multiple zeta values, let us introduce the $\mathbb{Q}$-vector space 
\[ \Tcal=\sum_{k=0}^\infty\,\Tcal_k\] spanned by all MTVs, where
\[  \Tcal_0=\mathbb{Q},\quad \Tcal_1=\{0\},\quad
\Tcal_{k}=\sum_{1\leq r \leq k-1 \atop {{k_1,\ldots,k_{r-1}\geq 1,\ k_r\geq 2} \atop k_1+\cdots+k_r=k}}
\mathbb{Q}\cdot T(k_1,\ldots,k_r)\quad (k\geq 2). 
\]
%Denote the $\mathbb{Q}$-dimension of $\Tcal_k$ by $d_k^{T}$ $(k=0,1,2,\ldots)$. 
The space $\Tcal$ becomes a $\Q$-algebra, the product of two MTVs being described by the 
{\it shuffle product}.  This is clear from the following integral expression of MTVs,
which is exactly parallel to that of multiple zeta values.

For a given tuple of numbers $\varepsilon_i \in \{0,1\}$ ($1\le i\le k$) with $\varepsilon_1=1$
and $\varepsilon_k=0$, set
\begin{equation*}
I(\varepsilon_1,\ldots,\varepsilon_k)=\mathop{\int\cdots\int}\limits_{0<t_1<\cdots <t_k<1}\varOmega_{\varepsilon_1}(t_1)\cdots \varOmega_{\varepsilon_k}(t_k),
\end{equation*}
where
\[\varOmega_0(t)=\frac{dt}{t},\quad \varOmega_1(t)=\frac{2\,dt}{1-t^2}.\]

Recall that an index set ${\bf k}=(k_1,k_2,\ldots,k_r)\in\mathbb{N}^r$ is {\it admissible} if $k_r\ge2$.
This ensures the convergence of the series \eqref{MTV} (as well as \eqref{hoft} and \eqref{mzv}).

\begin{theorem}\label{Th-3-1}\ For any admissible index set $(k_1,k_{2},\ldots,k_{r})$, we have
\begin{equation}
\TT(k_1,k_{2},\ldots,k_{r})=I(1,\underbrace{0,\ldots,0}_{k_1-1},1,\underbrace{0,\ldots,0}_{k_{2}-1},\cdots,1,\underbrace{0,\ldots,0}_{k_r-1}). \label{integ-exp}
\end{equation}
\end{theorem}
This can be seen by expanding $1/(1-t^2)$ into geometric series and integrate from left to right,
just as in the standard iterated integral expression of the multiple zeta value \eqref{mzv},
which is given in exactly the same form with $\varOmega_1(t)$ replaced by $dt/(1-t)$.  We should 
remark that this integral expression~\eqref{integ-exp} is essentially given in \cite{Sasaki2012}, 
although one needs some change of variables to obtain the current form.

From this integral expression, we immediately see that the 
{\em same} shuffle product rule holds for MTVs as for MZVs, an example being
\[ T(2)^2=4T(1,3)+2T(2,2)\quad \text{and}\quad \zeta(2)^2=4\zeta(1,3)+2\zeta(2,2). \]

Another immediate consequence of the integral expression is the duality.  We state and prove this in the 
next section, but remark here that the formula is again exactly the same as the duality formula of 
ordinary multiple zeta values. \\

Returning to the space $\Tcal$, the first question would be the dimension $d_k^T$ 
over $\Q$ of each subspace $\Tcal_k$ of weight $k$ elements.  
We have conducted numerical experiments with Pari-GP,
and obtained the following conjectural table.
\vspace{7pt}

\begin{center}
\begin{tabular}{|c|r|r|r|r|r|r|r|r|r|r|r|r|r|r|r|r|r|} \hline
$k$ & 0&1&2& 3& 4& 5& 6& 7& 8& 9& 10& 11& 12& 13& 14& 15\\ 
\hline  &&\\[-13pt]\hline
$d_k^T$ & 1& 0& 1& 1& 2& 2& 4& 5& 9& 10& 19& 23& 42& 49& 91& 110\\ 
\hline
\end{tabular}
\end{center}
\vspace{7pt}

Interestingly enough, the Fibonacci-like relation $d_k^T=d_{k-1}^T+d_{k-2}^T$ can be read off 
from the table for {\it even} $k$, but no immediate pattern is recognizable for general $k$.  
Recall that the conjectural dimension of the space of
alternating MZVs of weight $k$ (spanned by the numbers \eqref{eulersum} with all 
possible signs and $k_i$'s with $k_1+\cdots+k_r=k$, with $r$ varying) is given by the Fibonacci number
$F_k$ with $F_0=F_1=1$ and $F_k=F_{k-1}+F_{k-2}$.

We also note that Hoffman in \cite{Hoffman16} conjectures that the dimension
$d_k^t$ of the space spanned by his $t$-values \eqref{hoft} of weight $k$ is given
by the Fibonacci number $F_{k-1}$ (for $k\ge2$).
\vspace{7pt}

\begin{center}
\begin{tabular}{|c|r|r|r|r|r|r|r|r|r|r|r|r|r|r|r|r|r|} \hline
$k$ & 0&1&2& 3& 4& 5& 6& 7& 8& 9& 10& 11& 12& 13& 14& 15\\ 
\hline   &&\\[-13pt]\hline 
$d_k^t$ & 1& 0& 1& 2& 3& 5& 8& 13& 21& 34& 55& 89& 144& 233& 377& 610\\ 
\hline
$F_k$ & 1& 1 & 2& 3& 5& 8& 13& 21& 34& 55& 89& 144& 233& 377& 610 & 987\\ 
\hline
\end{tabular}
\end{center}
\vspace{7pt}

In \S\ref{relTt}, we present some speculations on relations between multiple $T$-values, 
Hoffman's $t$-values, and multiple zeta values.

%%%%%%%%%%%%%%%%%%%%%%%%%%%%%%%%%%%%%%%%%%%%
\section{Several identities among $T$-values}\label{sec-5}
%%%%%%%%%%%%%%%%%%%%%%%%%%%%%%%%%%%%%%%%%%%

In this section, we describe some formulas we have obtained so far.

\subsection{Duality}   Denote by ${\bf k}^\dagger$ the usual dual index of an admissible index set ${\bf k}$.
We assume the reader is familiar with the precise definition of the dual index and omit it here. 
See for instance the textbook of Zhao \cite{Zh}.

\begin{theorem}\label{Th-3-1}\ For any admissible index set ${\bf k}$, we have
\begin{equation}\label{duality}
\TT({\bf k}^\dagger)=\TT({\bf k}).
\end{equation}
\end{theorem}

\begin{proof}
The involution $t\to (1-t)/(1+t)$, which interchanges the differential forms
$\varOmega_0(t)$ and $\varOmega_1(t)$ and sends the interval $(0,1)$ to itself (with opposite orientation),
plays the role for the involution $t\to 1-t$ in the case of MZVs. That is to say, the change of variables 
\[ s_i=\frac{1-t_{k-i+1}}{1+t_{k-i+1}}\ \ (1\le i\le k) \]
in \eqref{integ-exp} immediately gives
\[ I(\varepsilon_1,\ldots,\varepsilon_k)=I(1-\varepsilon_k,\ldots,1-\varepsilon_1), \]
which is the required duality.
\end{proof}

\subsection{Sum formulas}  
For multiple zeta values, the classical sum formula is widely known and its variants 
are enormous (see \cite[Chapter~5]{Zh} for some of them). For our $T$-values, we only have certain formulas in depths 2 and 3.  In depth~2, we obtain an analogue of the weighted sum formula of Ohno and Zudilin \cite{OZ},
but in depth~3, we only obtain a formula which looks incomplete to be called as a sum formula.

\begin{theorem}\label{Th-sum1} \ 
For $k\in \mathbb{Z}_{\geq 3}$, we have
\begin{equation}
\sum_{j=2}^{k-1} 2^{j-1}\,\TT(k-j,j)=(k-1)\TT(k). \label{SF-2}
\end{equation}
\end{theorem}

\begin{theorem}\label{Th-sum2}\ For $k\in \mathbb{Z}_{\geq 4}$,
\begin{equation} \sum_{a+b+c=k \atop a, b\ge 1, c\ge 2}\TT(a,b,c)+\sum_{j=2}^{k-2}\TT(1,k-1-j,j)=\frac{2}{3}\TT(2)\TT(k-2).  \label{SF-3}
\end{equation}
\end{theorem}

We give proofs of these two theorems in the next section, and also present a conjectural (weighted) 
sum formula in depth 3.

\subsection{Parity result}

The so-called \textit{parity result}, proved in the case of  MZVs  in \cite{IKZ,Tsu2004}, 
also holds for MTVs.

\begin{theorem}\label{Th-5-2}\ Let ${\bf k}=(k_1,\ldots,k_{r})$ be an admissible index and assume 
its depth $r$ and weight $k_1+\cdots+k_{r}$ are of different parity. 
Then $\TT ({\bf k})$ can be expressed as a $\mathbb{Q}$-linear combination of multiple $\TT$-values
of lower depths and products of multiple $\TT$-values with sum of depths not exceeding $r$.
\end{theorem}

This was essentially proved in a previous paper \cite{Tsu2007} of the second named author.
Actually, what we have shown there was a reduction of $T$-values having depth and weight of opposite parity 
into a mixture of $T$-values and a certain multiple $L$-values with the character of conductor~$4$ of lower depth. 
But by checking carefully the proof of \cite[Theorem~1]{Tsu2007},
we see that Theorem~\ref{Th-5-2} is in principle already proved there.  We plan to write a detailed proof
separately in \cite{Tsu-parity}.

\begin{example}\label{Exam-5-3} For the case of depth $2$, we obtain the following formulas.
For $p\ge1$ and $q\ge2$ with $p+q$ odd, we have
\begin{align}
(-1)^q\,\TT(p,q) & =\binom{p+q-1}{q}\TT(p+q)\label{eq-5-1}\\
& \quad -\sum_{\mu=1\atop \mu\equiv q\bmod2}^{q-2}\binom{p+\mu-1}{\mu}
\frac{1}{2^{q-\mu}-1}\TT(p+\mu)\TT(q-\mu)\notag\\
& \quad -\sum_{\mu=0\atop \mu\equiv p\bmod2}^{p-2}
\binom{q+\mu-1}{\mu}\TT(p-\mu)\TT(q+\mu). \notag
\end{align}

We discuss a bit of a special case in depth 3 in \S\ref{relTt}.
%
%
%Furthermore, by the same method as in the proof of \cite[Section 4]{Tsu2007}, we can explicitly compute the triple case, for example,
%\begin{align*}
%\TT(1,1,2) & = \TT(4),   \\
%\TT(1,1,4) & = -\TT(2)\TT(1,3)+\frac{1}{9}\TT(2)^3,     \\
%\TT(1,2,3) & = \TT(3)^2-\TT(2)\TT(2,2),     \\
%\TT(1,3,2) & = - 2\TT(3)^2 +\TT(2)\TT(2,2)+3\TT(2)\TT(1,3),     \\
%\TT(2,1,3) & = -2\TT(3)^2+\TT(2)\TT(2,2) + 3\TT(2)\TT(1,3),   \\
%\TT(2,2,2) & = 4\TT(3)^2-\TT(2)\TT(2,2)-8\TT(2)\TT(1,3),     \\
%\TT(3,1,2) & = -\frac{1}{2}\TT(2)\TT(2,2) +\frac{1}{4}\TT(2)^3. 
%\end{align*}

\end{example}

\subsection{Height one $T$-values}

It is well-known (\cite{Ao, Dr}, see also \cite{Ohno-Zagier}) that the generating function of the `height one' multiple zeta values
is given in terms of the gamma function:
\[ 1-\sum_{m,n=1}^\infty \zeta(\underbrace{1,\ldots,1}_{n-1},m+1)X^m Y^n
=\frac{\Gamma(1-X)\Gamma(1-Y)}{\Gamma(1-X-Y)}.\] 
We can give the following $T$-version of this formula.

\begin{theorem}\label{Th-2-2}  We have the generating series identity
\[  1-\!\!\sum_{m,n=1}^\infty \TT(\underbrace{1,\ldots,1}_{n-1},m+1)X^m Y^n=
\frac{2\,\Gamma(1-X)\Gamma(1-Y)}{\Gamma(1-X-Y)}F(1-X,1-Y;1-X-Y;-1), \]
%\begin{align*}
%& \frac{2\Gamma(1-X)\Gamma(1-Y)}{\Gamma(1-X-Y)}F(1-X,1-Y;1-X-Y;-1)\\
%& \quad =1-\sum_{m,n=1}^\infty \TT(\overbrace{1,\ldots,1}^{n-1},m+1)X^m Y^n,
%\end{align*}
where $F(a,b;c;z)$ is the Gauss hypergeometric function and we assume $|X|<1, -1<Y<0$.
\end{theorem} 

\begin{proof}
From the integral expression \eqref{integ-exp}, we have
\begin{align*}
\TT(\underbrace{1,\ldots,1}_{n-1},m+1)&=\mathop{\int\cdots\int}\limits_{0<t_1<\cdots<t_n<u_1<\cdots<u_m<1}
\frac{2dt_1}{1-t_1^2}\,\cdots\,\frac{2t_n}{1-t_n^2}\,\frac{du_1}{u_1}\,\cdots\,\frac{du_m}{u_m} \\
&=\int_{0<t_n<1} \frac1{(n-1)!}\left(\int_0^{t_n} \frac2{1-t^2}\,dt\right)^{n-1}
\frac1{m!}\left(\int_{t_n}^1\frac1u\,du\right)^m\,\frac{2dt_n}{1-t_n^2}\\
%& =\int_0^1 \frac{2dt_n}{1-t_n^2}\int_0^{t_n}\frac{1+t_{n-1}^2}{1-t_{n-1}^2}\cdots \int_0^{t_2}\frac{2dt_1}{1-t_1^2}\int_{t_n}^{1}\frac{du_1}{u_1}\cdots \int_{u_{m-1}}^{1}\frac{du_m}{u_m}\\
& =\frac{1}{(n-1)!\,m!}\int_0^{1}\left\{\log\left( \frac{1+t_n}{1-t_n}\right)\right\}^{n-1}
\left\{\log\left( \frac{1}{t_n}\right)\right\}^{m}\frac{2dt_n}{1-t_n^2}.
\end{align*}
Hence we have
\begin{align*}
& \sum_{m,n=1}^\infty \TT(\underbrace{1,\ldots,1}_{n-1},m+1)X^m Y^{n-1}\\
& \quad = \int_0^1 \left( \frac{1+t}{1-t}\right)^Y(t^{-X}-1)\frac{2dt}{1-t^2}\\
& \quad = 2\int_0^1 t^{-X}(1-t)^{-Y-1}(1+t)^{Y-1}dt -\int_0^1 \left( \frac{1+t}{1-t}\right)^Y\frac{2dt}{1-t^2}.
\end{align*}
Denote the two integrals on the last line by $I_1$ and $I_2$ respectively. It follows from
the Euler integral
$$F(a,b;c;z)=\frac{\Gamma(c)}{\Gamma(a)\Gamma(c-a)}\int_0^1 t^{a-1}(1-t)^{c-a-1}(1-zt)^{-b}dt\quad (0<\Re a <\Re c)$$
that 
$$I_1=\frac{\Gamma(1-X)\Gamma(-Y)}{\Gamma(1-X-Y)}F(1-X,1-Y;1-X-Y;-1).$$
As for $I_2$, setting $w=\log\bigl((1+t)/(1-t)\bigr)$, 
we have 
$$I_2=\int_0^\infty e^{Yw}dw=-\frac{1}{Y}\quad (\text{if}\ Y<0).$$
Thus, multiplying $-Y$ and using $(-Y)\Gamma(-Y)=\Gamma(1-Y)$,  we obtain the desired formula.
\end{proof}

%%%%%%%%%%%%%%%%%%%%%%%%%%%%%%%%%%%%%%%%%%%%
\section{Proofs of the sum formulas}
%%%%%%%%%%%%%%%%%%%%%%%%%%%%%%%%%%%%%%%%%%%

In this section, we prove Theorems~\ref{Th-sum1} and \ref{Th-sum2}, and give a conjectural sum formula
for depth 3.

\subsection{Proof of Theorem~\ref{Th-sum1}}

We use two formulas of the function
\begin{align}
& \psi(k_1,\ldots,k_r;s) =\frac{1}{\Gamma(s)}\int_0^\infty t^{s-1}\frac{\AA(k_1,\ldots,k_{r};\tanh (t/2))}{\sinh(t)}\,dt\label{ee-6-1}
\end{align}
which was studied in our previous paper \cite{KT2}.  Here, $\AA(k_1,\ldots,k_r;z)$ is 
given by 
\[
\AA(k_1,\ldots,k_r;z)=2^r\sum_{0<m_1<\cdots <m_r\atop m_i\equiv i\bmod 2}
\frac{z^{m_r}}{m_1^{k_1}\cdots m_{r}^{k_{r}}}.
\]
(In \cite{KT2}, $2^{-r}\AA(k_1,\ldots,k_r;z)$ is denoted by ${\rm Ath}(k_1,\ldots,k_r;z)$.)
The formulas we need are special cases of \cite[Theorems~5.3 and 5.5]{KT2}, which read
(by letting $k=2$, $r\to k-2$ and $m=0$)
\begin{align}
& \psi(\underbrace{1,\ldots,1}_{k-3},2;s)  \label{psi}\\
&\quad =-\sum_{j=2}^{k-1}
\binom{s+j-2}{j-1}\,\TT(k-j,j-1+s)-\TT(k-1,s)+\TT(k-1)\TT(s) \notag
\end{align}
and 
\begin{equation}\label{psivalue} 
\psi(\underbrace{1,\ldots,1}_{k-3},2;1)=\TT(1,k-1).
\end{equation}
We also use the fact that the function $\psi(\underbrace{1,\ldots,1}_{k-3},2;s)$ is holomorphic everywhere.
Since each function $\TT(k-j,j-1+s)$ in the sum on the right of \eqref{psi} is holomorphic at $s=1$, 
the remaining sum $-\TT(k-1,s)+\TT(k-1)\TT(s)$ should be holomorphic at $s=1$ (each of 
$\TT(k-1,s)$ and $\TT(k-1)\TT(s)$ has a pole of order 1 at $s=1$). To evaluate 
the value of $-\TT(k-1,s)+\TT(k-1)\TT(s)$ at $s=1$, we compute the `stuffle 
product' 
\begin{align}\label{zetaeo}
&\frac12\,T(k-1)\cdot 2^{-s}\zeta(s) \\
&=\sum_{m=1\atop m:\text{odd}}\frac1{m^{k-1}}\,\sum_{n=2\atop n:\text{even}}\frac1{n^s}
=\sum_{0<m<n\atop m:\text{odd,}\,n:\text{even}}\frac1{m^{k-1}n^s}+
\sum_{0<n<m\atop n:\text{even,}\,m:\text{odd}}\frac1{n^s m^{k-1}}\notag \\ 
&=\frac14\,T(k-1,s)+\zeta^{eo}(s,k-1), \notag
\end{align}
($\zeta^{eo}(s,k-1)$ is the last sum in \eqref{zetaeo})
from which we have 
\begin{align*}  
&-\TT(k-1,s)+\TT(k-1)\TT(s) \\
&\quad = 4\zeta^{eo}(s,k-1)- 2T(k-1)\cdot 2^{-s}\zeta(s)+\TT(k-1)\TT(s)  \\
&\quad = 4\zeta^{eo}(s,k-1)- 2T(k-1)\cdot 2^{-s}\zeta(s)+\TT(k-1)\cdot 2(1-2^{-s})\zeta(s)\\
&\quad = 4\zeta^{eo}(s,k-1)+\TT(k-1)\cdot 2(1-2^{1-s})\zeta(s).
\end{align*}
We then see that $\zeta^{eo}(s,k-1)$ is finite at $s=1$ and so is
\[ (1-2^{1-s})\zeta(s)=1-\frac1{2^s}+\frac1{3^s}-\frac1{4^s}+\cdots \]
whose value at $s=1$ is $\log2$.  Hence we have
\[ \lim_{s\to 1} \left(-\TT(k-1,s)+\TT(k-1)\TT(s)\right)=4\zeta^{eo}(1,k-1)+
(2\log2)\TT(k-1).  \]

To compute the value $\zeta^{eo}(1,k-1)$, we consider the specific alternating MZV
\begin{equation}
\sigma_a(1,k-1)=\sum_{1\leq m<n}\frac{(-1)^{m-1}}{m n^{k-1}}=\left(1-2^{-k+1}\right)\zeta(1,k-1)
-2{\zeta}^{eo}(1,k-1).\label{eq-4-1}
\end{equation}
We use the formula by Borwein et al \cite[\S 4]{Borwein95} (we are using the notation there)
\begin{align*}
\sigma_a(1,k-1)&=(2\log 2)(1-2^{-k+1})\zeta(k-1)-\frac{k-1}2\zeta(k)\\
&+\frac12\sum_{j=2}^{k-2}(1-2^{1-j})(1-2^{j-k+1})\zeta(j)\zeta(k-j)
\end{align*}
and by Euler
\begin{equation}\label{Euler} 
\zeta(1,k-1)=\frac{k-1}2\zeta(k)-\frac12\sum_{j=2}^{k-2} \zeta(j)\zeta(k-j) 
\end{equation} 
to conclude
\begin{align*}
4{\zeta}^{eo}(1,k-1)&=2(1-2^{-k+1})\zeta(1,k-1)-2\sigma_a(1,k-1)\\
&=(2-2^{-k+1})(k-1)\zeta(k)-(4\log 2)(1-2^{-k+1})\zeta(k-1)\\
& \ -\sum_{j=2}^{k-2}\left\{(1-2^{-k+1})+(1-2^{1-j})(1-2^{j-k+1})\right\}\zeta(j)\zeta(k-j)\\
& =(k-1)\TT(k)-(2\log 2)\TT(k-1)-\frac{1}{2}\sum_{j=2}^{k-2}\TT(j)\TT(k-j).
\end{align*}
We have used $2(1-2^{-m})\zeta(m)=\TT(m)$ and 
\[(1-2^{-k+1})+(1-2^{1-j})(1-2^{j-k+1})=2(1-2^{-j})(1-2^{-k+j}).\]
We therefore have 
\[
\lim_{s\to 1} \left(-\TT(k-1,s)+\TT(k-1)\TT(s)\right)=(k-1)\TT(k)
-\frac{1}{2}\sum_{j=2}^{k-2}\TT(j)\TT(k-j),  \]
and by letting $s\to1$ in \eqref{psi} together with \eqref{psivalue} we obtain
\begin{equation}
\sum_{j=2}^{k-1} \TT(k-j,j)+\TT(1,k-1) =(k-1)\TT(k)-\frac{1}{2}\sum_{j=2}^{k-2} \TT(j)\TT(k-j). \label{interm}
\end{equation}
Now, recall the shuffle product expansion of $\TT(j)\TT(k-j)$ has the same form as that of 
$\zeta(j)\zeta(k-j)$ given in {\em e.g.} \cite[p.~72, (3)]{GKZ}, which is
$$\TT(j)\TT(k-j)=\sum_{\nu=2}^{k-1}\left\{\binom{\nu-1}{j-1}+\binom{\nu-1}{k-j-1}\right\}\TT(k-\nu,\nu).$$
Summing up, we obtain
\begin{align}\label{sumprodT}
\frac12\sum_{j=2}^{k-2}\TT(j)\TT(k-j)&=\frac12\sum_{\nu=2}^{k-1}\left(\sum_{j=2}^{k-2}\left\{\binom{\nu-1}{j-1}
+\binom{\nu-1}{k-j-1}\right\}\right)\TT(k-\nu,\nu) \\
&=\sum_{\nu=2}^{k-1}\left(\sum_{j=2}^{k-2}\binom{\nu-1}{j-1}\right)\TT(k-\nu,\nu)\notag\\
&=\sum_{\nu=2}^{k-2}(2^{\nu-1}-1)\TT(k-\nu,\nu)+(2^{k-2}-2)\TT(1,k-1).\notag
\end{align}
Here, we have used
\begin{equation*}
\sum_{j=2}^{k-2}\binom{\nu-1}{j-1}=
\begin{cases}
2^{\nu-1}-1 & (\nu\leq k-2),\\
2^{k-2}-2 & (\nu=k-1).
\end{cases}
\end{equation*}
Combining \eqref{interm} and \eqref{sumprodT}, we obtain Theorem~\ref{Th-sum1}.

\begin{remark}\label{Remark-4-7}\ 
The weighted sum formula for the double zeta values is
\[ \sum_{j=2}^{k-1} 2^{j-1}\,\zeta(k-j,j)=\frac{k+1}{2}\zeta(k). \]
This can also be proved in the same manner as in the above last step by starting with Euler's \eqref{Euler}
and expressing $\zeta(j)\zeta(k-j)$ as a sum of double zeta values by the shuffle 
product, and by using the ordinary sum formula.
\end{remark}

%%%%%%%%%%%%%%%%%%%%%%%%%%%%%%%%%%%%%%%%%%%%
\subsection{Triple $\TT$-values}\label{sec-5-2}
%%%%%%%%%%%%%%%%%%%%%%%%%%%%%%%%%%%%%%%%%%%

The method of proof here is different from that in the previous subsection and uses 
partial fraction decompositions.
We start with a lemma.

\begin{lemma}\label{Lemma-4-8}\ For $q\in \mathbb{N}$, it holds
\begin{align}
&\sum_{l=-\infty}^\infty \sum_{m,n= 0}^\infty \frac{1}{(2l+1)(2m+1)(2n+1)^q(2l+2m+2n+3)}=0. \label{eq-4-2}
\end{align}
\end{lemma}

\begin{proof}
It is well-known that
$$\sum_{l=0}^\infty \frac{\sin((2l+1)x)}{2l+1}=\frac{\pi}{4},$$
which is uniformly convergent for $0<x<\pi$ 
(see \cite[\S 2.2]{Titch}). Setting $x=\pi/2+\theta$, we have
\begin{equation*}
\lim_{L\to \infty}\sum_{l=-L}^L \frac{(-1)^l e^{(2l+1)i\theta}}{2l+1}=\frac{\pi}{2}\quad \left(-\frac{\pi}{2}<\theta<\frac{\pi}{2}\right), %\label{eq-4-2}
\end{equation*}
where $i=\sqrt{-1}$. For simplicity we write the limit of the left-hand side as $\sum_{l=-\infty}^{\infty}$.
Hence, for $q\in \mathbb{N}$ and $z\in (-1,1)$, we have
\begin{align*}
0=&\left(\sum_{l=-\infty}^\infty \frac{(-1)^l e^{(2l+1)i\theta}}{2l+1}-\frac{\pi}{2}\right)\sum_{m,n=0}^\infty \frac{(-z)^{m+n}e^{(2m+2n+2)i\theta}}{(2m+1)(2n+1)^q}\\
&=\sum_{l=-\infty}^\infty \frac{(-1)^l}{2l+1}\sum_{m,n= 0}^\infty  \frac{(-z)^{m+n}e^{(2l+2m+2n+3)i\theta}}{(2m+1)(2n+1)^q}-\frac{\pi}{2}\sum_{m,n=0}^\infty \frac{(-z)^{m+n}e^{(2m+2n+2)i\theta}}{(2m+1)(2n+1)^q}
\end{align*}
for $-\pi/2<\theta<\pi/2$. 
Integrating the both sides from $-t$ to $t$ $(-\pi/2<t<\pi/2)$, we obtain
\begin{align*}
0=&2 \sum_{l=-\infty}^\infty \frac{(-1)^l}{2l+1}\sum_{m,n= 0}^\infty  \frac{(-z)^{m+n}\sin((2l+2m+2n+3)t)}{(2m+1)(2n+1)^q(2l+2m+2n+3)}\\
& \quad -\pi\sum_{m,n=0}^\infty \frac{(-z)^{m+n}\sin((2m+2n+2)t)}{(2m+1)(2n+1)^q(2m+2n+2)}.
\end{align*}
We can easily see that the right-hand side is absolutely and uniformly convergent for $t\in [-\pi/2,\pi/2]$ and $z\in [-1,1]$. Hence, letting $t\to \pi/2$ and $z\to 1$, we obtain \eqref{eq-4-2}.
\end{proof}

We can write \eqref{eq-4-2} as 
\begin{align}
&\sum_{l,m,n\geq 0} \frac{1}{(2l+1)(2m+1)(2n+1)^q(2l+2m+2n+3)}\label{eq-4-3} \\
& \quad -\sum_{l,m,n\geq 0} \frac{1}{(2l+1)(2m+1)(2n+1)^q(-2l+2m+2n+1)}=0 .\notag
\end{align}
Denote the sums on the left-hand side by $I_1$ and $I_2$ respectively, so that $I_1=I_2$ holds. 
We shall write down each $I_1$ and $I_2$ in terms of $\TT$-values, by using the following partial fraction decomposition formulas.

\begin{lemma}\label{Lemma-4-9}\ For $q\in \mathbb{N}$,
\begin{align}
\frac{1}{xy^q}&=\sum_{j=0}^{q-1}\frac{1}{y^{q-j}(x+y)^{j+1}}+\frac{1}{x(x+y)^q}, \label{eq-4-4}\\
\frac{1}{xyz^q}&=\sum_{j=0}^{q-1}\bigg\{\sum_{\nu=0}^{q-j-1}\frac{1}{z^{q-j-\nu}(x+z)^{\nu+1}(x+y+z)^{j+1}}+\frac{1}{x(x+z)^{q-j}(x+y+z)^{j+1}}\label{eq-4-5}\\
& \qquad +\sum_{\nu=0}^{q-j-1}\frac{1}{z^{q-j-\nu}(y+z)^{\nu+1}(x+y+z)^{j+1}}+\frac{1}{y(y+z)^{q-j}(x+y+z)^{j+1}}\bigg\} \notag\\
& \quad +\frac{1}{x(x+y)(x+y+z)^q}+\frac{1}{y(x+y)(x+y+z)^q}. \notag
\end{align}
\end{lemma}

\begin{proof}
Equation \eqref{eq-4-4} immediately follows from the factorization
\[ \frac1{y^q} -\frac1{(x+y)^q}=\left(\frac1y-\frac1{x+y}\right)
\sum_{j=0}^{q-1}\frac1{y^{q-1-j}(x+y)^j}=\frac{x}{y(x+y)}\sum_{j=0}^{q-1}\frac1{y^{q-1-j}(x+y)^j}. \]
Replacing $y$ by $z$ and $x$ by $x+y$ in \eqref{eq-4-4} and then multiplying $(x+y)/xy=1/x+1/y$, we have
$$
\frac{1}{xyz^q}=\sum_{j=0}^{q-1}\left(\frac{1}{xz^{q-j}}+\frac{1}{yz^{q-j}}\right)\frac{1}{(x+y+z)^{j+1}}+\frac{1}{xy(x+y+z)^q}. 
$$
Applying \eqref{eq-4-4} to $1/xz^{q-j}$ and $1/yz^{q-j}$ and writing $1/xy$ as $1/x(x+y)+1/y(x+y)$
in the last term, we obtain \eqref{eq-4-5}.
\end{proof}

\noindent {\it Proof of Theorem~\ref{Th-sum2}.}\quad
Using  \eqref{eq-4-5} with $x=2l+1, y=2m+1, z=2n+1$, we readily have
(note $2l+2m+2n+3=x+y+z$)
\begin{align}
I_1&=\frac14\sum_{j=0}^{q-1}\left\{\sum_{\nu=0}^{q-1-j}\TT(q-j-\nu,\nu+1,j+2)+
\TT(1,q-j,j+2)\right\}+\frac14\TT(1,1,q+1)\label{I1}\\
&=\frac14\sum_{a+b+c=q+3 \atop a, b\ge 1, c\ge 2}\TT(a,b,c)+\frac14\sum_{j=2}^{q+1}\TT(1,q+2-j,j)
+\frac14\TT(1,1,q+1).\notag
\end{align}
As for $I_2$, set $d=n-l$ or $e=l-n$ according as $l<n$ or $l\geq n$. Then 
\begin{align}
I_2=&\sum_{d\geq 1 \atop l,m\geq 0} \frac{1}{(2l+1)(2m+1)(2d+2l+1)^q(2d+2m+1)}\label{I2}\\
& +\sum_{e,m,n\geq 0} \frac{1}{(2e+2n+1)(2m+1)(2n+1)^q(-2e+2m+1)}.\notag
\end{align}
The first sum on the right is equal to
\begin{align*}
&\sum_{d\geq 1 \atop l\geq 0} \frac{1}{(2l+1)(2d+2l+1)^q}\,\frac1{(2d)}\sum_{m=0}^\infty \left(\frac{1}{2m+1}-\frac{1}{2m+2d+1}\right)\\
&=\sum_{d\geq 1, l\geq 0 \atop 0\le m\le d-1} \frac{1}{(2l+1)(2d+2l+1)^q(2d)(2m+1)}\\
& =\sum_{l,m,k\geq 0} \frac{1}{(2l+1)(2m+1)(2m+2k+2)(2l+2m+2k+3)^q}\quad (d=m+k+1)\\
& = \sum_{l,m,k\geq 0} \frac{1}{(2l+1)(2m+1)(2l+2m+2k+3)^{q+1}}\\
& \qquad +\sum_{l,m,k\geq 0} \frac{1}{(2m+1)(2m+2k+2)(2l+2m+2k+3)^{q+1}}\\
&=\sum_{l,m,k\geq 0} \frac{1}{(2l+1)(2l+2m+2)(2l+2m+2k+3)^{q+1}}\\
&\qquad+\sum_{l,m,k\geq 0} \frac{1}{(2m+1)(2l+2m+2)(2l+2m+2k+3)^{q+1}}\\
& \qquad\qquad +\sum_{l,m,k\geq 0} \frac{1}{(2m+1)(2m+2k+2)(2l+2m+2k+3)^{q+1}}\\
&=\frac38\,\TT(1,1,q+1).
\end{align*}
The second sum in \eqref{I2} is, by setting $f=m-e$ or $g=e-m$ according as 
$e\leq m$ or $e>m$, transformed into
\begin{align*}
&\sum_{e,f,n\geq 0} \frac{1}{(2e+2n+1)(2e+2f+1)(2n+1)^q(2f+1)}\\
& +\sum_{g\geq 1 \atop m,n\geq 0} \frac{1}{(2g+2m+2n+1)(2m+1)(2n+1)^q(-2g+1)}.
\end{align*}
The second sum of this is equal to $-I_1$ as seen by setting $g=l+1$. We write the first sum,
first by separating the terms with $e=0$ and $e>0$, as
\begin{align*}
&\sum_{f,n\geq 0} \frac{1}{(2f+1)^2(2n+1)^{q+1}}
+\sum_{e\geq 1, f,n\geq 0} \frac{1}{(2e+2n+1)(2n+1)^q(2e+2f+1)(2f+1)}\\
&=\frac14 \TT(2)\TT(q+1) +\sum_{e\geq 1 \atop n\geq 0} \frac{1}{(2e+2n+1)(2n+1)^q}\,\frac1{(2e)}
\sum_{f=0}^\infty \left(\frac{1}{2f+1}-\frac{1}{2f+2e+1}\right)\\
&=\frac14 \TT(2)\TT(q+1) +\sum_{e\geq 1, n\geq 0 \atop 0\le f\le e-1} 
\frac{1}{(2e+2n+1)(2n+1)^q(2e)(2f+1)}\\
& =\frac14 \TT(2)\TT(q+1) 
+\sum_{f,l,n\geq 0}\frac{1}{(2f+1)(2f+2l+2)(2n+1)^q(2f+2l+2n+3)}
\end{align*}
($e=f+l+1$).  Using \eqref{eq-4-4} repeatedly, we have
\begin{align*}
& \sum_{f,l,n\geq 0}\frac{1}{(2f+1)(2f+2l+2)(2n+1)^q(2f+2l+2n+3)}\\
& = \sum_{f,l,n\geq 0}\bigg\{ \sum_{j=0}^{q-1}\frac{1}{(2f+1)(2n+1)^{q-j}(2f+2l+2n+3)^{j+2}}\\
&\qquad +\frac{1}{(2f+1)(2f+2l+2)(2f+2l+2n+3)^{q+1}}\bigg\}\\
& = \sum_{f,l,n\geq 0}\sum_{j=0}^{q-1}\bigg\{ \sum_{\nu=0}^{q-j-1}\frac{1}{(2n+1)^{q-j-\nu}(2f+2n+2)^{\nu+1}
(2f+2l+2n+3)^{j+2}}\\
&\qquad\qquad +\frac1{(2f+1)(2f+2n+2)^{q-j}(2f+2l+2n+3)^{j+2}}\biggr\}\\
&\qquad +\frac18 \TT(1,1,q+1)\\
&=\frac18\sum_{j=0}^{q-1}\left\{\sum_{\nu=0}^{q-1-j}\TT(q-j-\nu,\nu+1,j+2)
+\TT(1,q-j,j+2)\right\} + \frac18\TT(1,1,q+1)\\
&=\frac18\sum_{a+b+c=q+3 \atop a, b\ge 1, c\ge 2}\TT(a,b,c)+\frac18\sum_{j=2}^{q+1}\TT(1,q+2-j,j)
+ \frac18\TT(1,1,q+1).
\end{align*}
We therefore have 
\begin{align*} 
I_2&=\frac38\,\TT(1,1,q+1)-I_1+ \frac14 \TT(2)\TT(q+1)\\
&+\frac18\sum_{a+b+c=q+3 \atop a, b\ge 1, c\ge 2}\TT(a,b,c)+\frac18\sum_{j=2}^{q+1}\TT(1,q+2-j,j)
+ \frac18\TT(1,1,q+1). 
\end{align*}
Combining this and \eqref{I1} together with $I_1=I_2$ and setting $q+3=k$ gives the theorem. \hfill \qed

\begin{example}\label{Ex-3-3}
The case $k=5$ of Theorem~\ref{Th-sum2} is
\begin{align}
& 2\TT(1,1,3)+2\TT(1,2,2)+\TT(2,1,2)=\frac{2}{3}\TT(2)\TT(3). \label{SFwt5}
\end{align}
This is not quite parallel to the case of ordinary MZVs, where the identity
\[ 2\zeta(1,1,3)+2\zeta(1,2,2)+\zeta(2,1,2)=2\zeta(2)\zeta(3)-\frac{5}{2}\zeta(5)  \]
holds.  It is unlikely that the right-hand side is a multiple of $\zeta(2)\zeta(3)$.
\end{example}

\begin{remark}\label{Rem-4-11}\ We may prove \eqref{interm} by a similar argument 
starting from 
\begin{align*}
0=&\left(\sum_{l=-\infty}^\infty \frac{(-1)^l e^{(2l+1)i\theta}}{2l+1}-\frac{\pi}{2}\right)\sum_{m=0}^\infty \frac{(-z)^{m}e^{(2m+1)i\theta}}{(2m+1)^q}.
\end{align*}
Hence Theorem~\ref{Th-sum2} can be regarded as a triple version of \eqref{interm}.
\end{remark}

We end this section by proposing the following conjecture as an analogue of 
Machide's formula \cite[Corollary~4.1]{M}.

\begin{conj}\label{machide}  For $k\ge4$, we have
\[ \sum_{a+b+c=k \atop a, b\ge 1, c\ge 2}2^b(3^{c-1}-1)\TT(a,b,c)=\frac23(k-1)(k-2)\TT(k). \]
\end{conj}

%%%%%%%%%%%%%%%%%%%%%%%%%%%%%%%%%%%%%%%%%%%%%%%%%%%%%%%%%%
\section{Relations among multiple $T$-, $t$-, and zeta values}\label{relTt}
%%%%%%%%%%%%%%%%%%%%%%%%%%%%%%%%%%%%%%%%%%%%%%%%%%%%%%%%%%

If we denote by $\mathcal{T}^*$ the $\Q$-vector space spanned by all Hoffman's multiple
$t$-values, then, as can be directly seen from the definition \eqref{hoft},  the space 
$\mathcal{T}^*$ also becomes a $\Q$-algebra by the {\it stuffle} (or {\it harmonic}) product, 
an example being $t(2)^2=2t(2,2)+t(4)$. Hence, we have two $\Q$-subalgebas $\Tcal$ and 
$\mathcal{T}^*$ of the algebra of alternating multiple zeta values, one being closed under the shuffle product
and the other under the stuffle product.  There are both shuffle and stuffle product structures
on the whole space of alternating multiple zeta values.

It seems that the sum $\Tcal+\mathcal{T}^*$ does not exhaust all alternating MZVs, and that the
seemingly smaller space $\Tcal$
is not contained in $\mathcal{T}^*$, as the following table (numerically computed, only up to
weight 8) suggests.
\vspace{7pt}
\begin{center}
\begin{tabular}{|c|r|r|r|r|r|r|r|r|r|} \hline
$k$ & 0&1&2& 3& 4& 5& 6& 7& 8\\ 
\hline   &&\\[-13pt]\hline 
$\dim(\Tcal_k+\mathcal{T}^*_k)$ & 1& 0& 1& 2& 4& 5& 9& 14& 24\\ 
\hline
$\dim(\Tcal_k\cap\mathcal{T}^*_k)$& 1& 0 & 1& 1& 1& 2& 3& 4& 6\\ 
\hline
\end{tabular}
\end{center}
\vspace{7pt}

Let $\ZZ$ be the space of usual multiple zeta values. The well-known conjectural dimension 
(Zagier~\cite{Z}) of the subspace $\ZZ_k$ of weight $k$ 
is given by the sequence $d_k$ which satisfies $d_k=d_{k-2}+d_{k-3}$ with $d_0=1, d_1=0, d_2=1$.
\vspace{7pt}

\begin{center}
\begin{tabular}{|c|r|r|r|r|r|r|r|r|r|r|r|r|r|r|r|r|r|} \hline
$k$ & 0&1&2& 3& 4& 5& 6& 7& 8& 9& 10& 11& 12& 13& 14& 15\\ 
\hline  &&\\[-13pt]\hline
$d_k$ & 1& 0& 1& 1& 1& 2& 2& 3& 4& 5& 7& 9& 12& 16& 21& 28\\ 
\hline
\end{tabular}
\end{center}
\vspace{7pt}

It appears that $d_k^T\ge d_k$ holds for all $k$ ($F_{k-1}\ge d_k$ is certainly true, where $F_{k-1}$
is conjectured to be equal to $d_k^t$).
Moreover, we conjecture (also based on our numerical experiments) that the 
space $\Tcal$ as well as $\mathcal{T}^*$ contains the space~$\ZZ$.

\begin{conj}  
Both $\Tcal$ and $\mathcal{T}^*$ contain $\ZZ$ as a $\Q$-subalgebra.
\end{conj}

The intersection $\Tcal\cap\mathcal{T}^*$ seems strictly larger than $\ZZ$, as the tables above 
suggest. If this conjecture is true, then both $\Tcal$ and $\mathcal{T}^*$ are modules over
$\ZZ$.  What are the structures of these modules?

Specific elements show an interesting pattern.
By definition, the single $T$-values $\TT(k)$ and $t$-values $t(k)$ are multiples of $\zeta(k)$ and
hence contained in $\ZZ$.  And by our parity result (Theorem~\ref{Th-5-2}), every double $T$-value 
of odd weight is also contained in $\ZZ$.
For higher depths, we conjecture the following. Since we have the duality for $T$-values, 
we may restrict ourselves to the $T$-values with depth smaller than or equal to $\text{weight}/2$.

\begin{conj}\label{conj5-2}
1) For even weights, other than the single $T$-value $\TT(k)$, only $T(p,q,r)$ with $p,r:\text{odd}\ge3$ and $q:\text{even}$ (and their duals) are in $\ZZ$.

2) If the weight is odd, other than the single and the double $T$-values,
only $T(p,1,r)$ with $p,r:\text{even}$ (and their duals) are in $\ZZ$.
\end{conj}

Recall that, from our parity result, the triple $T$-value $T(p,q,r)$ of even weight
can be written in terms of single and the double $T$-values.  From an explicit formula for such an expression
(see \cite{Tsu-parity} for the detail), we surmise that the following is true.

\begin{conj}\label{conj5-3} For $m\ge1, p\ge1, q\ge2$ with $p+q+m$ even, we have
\[ \sum_{i+j=m\atop i,j\ge0} \binom{p+i-1}{i}\binom{q+j-1}{j}T(p+i,q+j)\,\in\,\ZZ.\]
\end{conj}
For instance, the case $m=1$ predicts $qT(p,q+1)+pT(p+1,q)\in\ZZ$.

\begin{remark}  Denoting the sum in the conjecture above by $s(p,q,m)$,
the form of the parity reduction for $T(2p+1,2q,2r+1)$ is
\begin{align*}  
&T(2p+1,2q,2r+1)\\
&=-\sum_{j=0}^{p-1}T(2p-2j)s(2q-1,2r+1,2j+1)
-\sum_{j=0}^{r-1} T(2r-2j)s(2q,2j+2,2p)\\
&\quad +\text{sum of products of single $T(n)$'s}. 
\end{align*}
\end{remark}
\vspace{7pt}
 
As for $t$-values, we experimentally observe that any $t(k_1,\ldots,k_r)$ with $\forall k_i\ge2$ is in $\ZZ$.
Among those, we may choose the following elements as linear and algebraic bases of $\ZZ$.

\begin{conj}
1)  A linear basis of the space $\ZZ_k$ of multiple zeta values of weight $k$ is given by
\[ \{t(2)^n t(k_1,\ldots,k_r)\,\mid\, n,r\ge0, \forall k_i:\text{odd} \ge3,\, 2n+k_1+\cdots+k_r=k \} \]

2) An algebra basis of $\ZZ$ is given by $t(2)$ and $t(k_1,\ldots,k_r)$ with $\forall k_i:\text{odd}  \ge3$
and the sequence $(k_1,\ldots,k_r)$ being Lyndon. 
\end{conj}
With the usual order by magnitude, a sequence $(k_1,\ldots,k_r)$ is Lyndon if any right 
subsequence $(k_i,\ldots,k_r)$ ($i\ge2$) is greater than $(k_1,\ldots,k_r)$ in lexicographical order.

\begin{remark} 
Quite recently, T.~Murakami proved our observation $t(k_1,\ldots,k_r)\in\ZZ$ if $\forall k_i\ge2$,
by using the motivic method employed in \cite{Gla}. Also he could prove Conjecture~\ref{conj5-3}
or Conjecture~\ref{conj5-2} except the `only' parts.
\end{remark}

%%%%%%%%%%%%%%%%%%%%%%%%%%%%%%%%%%%%%%%%%%%%%%%%%%%%%%%%%%
\section{Description of the space $\Tcal_k$ for low weights} 
%%%%%%%%%%%%%%%%%%%%%%%%%%%%%%%%%%%%%%%%%%%%%%%%%%%%%%%%%%

Obviously $\Tcal_2=\Q\cdot T(2)$ is one dimensional, and by the duality \eqref{duality} the space 
\[ \Tcal_3=\Q\cdot T(3)+\Q\cdot T(1,2) =\Q\cdot T(3)\] is also one dimensional.

At weight 4, we have $T(1,1,2)=T(4)$ by the duality and $T(2,2)=\frac12 T(4)-2T(1,3)$ by 
the sum formula \eqref{SF-2}, and thus we see that 
\[ \Tcal_4=\Q\cdot T(4)+\Q\cdot T(1,3).  \]
According to our conjecture (see the table in \S2), this would give a basis of $\Tcal_4$.

We conjecture that the space $\Tcal_5$ of weight 5 is also two dimensional. By the duality,
we see that $\Tcal_5$ is spanned by $T(5)$ and elements of depth 2.  We have two independent
relations
\begin{align*}
& 4T(1,4)+2T(2,3)+T(3,2)=2T(5), \\
& 2 T(1, 4) + 2T(2,3) + T(3, 2)=4T(1,4)+2T(2,3)+\frac23 T(3,2)
\end{align*}
coming from \eqref{SF-2} and \eqref{SFwt5} (as for the latter, we used the duality on the left-hand side
and the shuffle product on the right), and from these we obtain
\[ T(3,2)=6T(1,4)\quad\text{and}\quad T(2, 3) =T(5)- 5T(1, 4). \]
Hence we conclude 
\[ \Tcal_5=\Q\cdot T(5)+\Q\cdot T(1,4). \]

Already at weight 6, known identities appear not to be enough to reduce the dimension to the 
conjectural 4. Using Theorems~\ref{Th-3-1} through \ref{Th-5-2} and relations obtained 
by applying the shuffle product to lower weight relations, we may deduce
\begin{align*}
T(1, 2, 3) &= -\frac{25}{12} T(6) + 12 T(1, 5) + 6 T(2, 4) + 2 T(3, 3) - 2 T(1, 1, 4), \\
T(1, 3, 2) &=  \frac{55}{12} T(6) - 24 T(1, 5) - 12 T(2, 4) - 4 T(3, 3) - T(1, 1, 4), \\
T(2, 1, 3) &=  \frac{55}{12} T(6) - 24 T(1, 5) - 12 T(2, 4) - 4 T(3, 3) -   T(1, 1, 4), \\
T(2, 2, 2) &= -\frac{35}4 T(6) + 48 T(1, 5) + 24 T(2, 4) +   8 T(3, 3) + 6 T(1, 1, 4), \\
T(3, 1, 2) &= \frac56 T(6) - T(1, 1, 4), \\
T(4, 2) &= \frac52 T(6) - 8 T(1, 5) - 4 T(2, 4) - 2 T(3, 3).
\end{align*}
One missing relation would be supplied by Conjecture~\ref{conj5-3}, which predicts for instance
\begin{align*} 3T(2,4)+2T(3,3)&=-\frac{15}7 T(6)+\frac{10}7 T(3)^2\\
&=-\frac{15}7 T(6) + \frac{120}7 T(1, 5)+ \frac{60}7 T(2, 4) + \frac{20}7 T(3, 3).
\end{align*}
(Note that the space of multiple zeta values of weight 6 is spanned by $\zeta(6)=\frac{32}{63} T(6)$ and
$\zeta(3)^2=\frac{16}{49}T(3)^2$.)  From this we could conclude 
\[ \Tcal_6=\Q\cdot T(6)+\Q\cdot T(1,5)+\Q\cdot T(2,4)+\Q\cdot T(1,1,4). \]

In a similar vein, we may deduce by using proven relations that the space $\Tcal_7$ is at most 6 dimensional,
and by assuming Conjecture~\ref{machide}, we may reduce the dimension to the 
conjectural 5.\\

Since it becomes more and more tedious to write down the parity reduction 
explicitly as the depth gets larger, we have not checked if all relations obtained and
conjectured in this paper are enough to give the conjectural upper bound of the 
dimension of $\Tcal_k$ for $k$ greater than 7.  Are there any other nicer families of 
relations among MTVs, and what is the complete set of linear relations?

\ 

{\bf Acknowledgements.}\ 
{This work was supported by Japan Society for the Promotion of Science, Grant-in-Aid for Scientific Research (S) 16H06336 (Kaneko) and (C) 18K03218 (Tsumura).}

\


\begin{thebibliography}{999}

\bibitem{Ao} K.~Aomoto, {Special values of hyperlogarithms and
  linear difference schemes}, Illinois J. of Math., {\bf 34-2} (1990),
  191--216.

\bibitem{AK1999}  T.~Arakawa and M.~Kaneko, 
Multiple zeta values, poly-Bernoulli numbers, and related zeta functions,  Nagoya Math. J., {\bf 153} (1999), 189--209.


\bibitem{Bachmann18}  H.~Bachmann, 
Modular forms and $q$-analogues of modified double zeta values, arXiv:1808.09674.


\bibitem{Borwein95} 
D.~Borwein, J.~M.~Borwein and R.~Girgensohn, Explicit evaluation of Euler sums, Proc. Edinburgh Math. Soc. {\bf 38} (1995), 277-294. 


\bibitem{Dr} V.~G.~Drinfel'd, {On quasitriangular quasi-Hopf
  algebras and a group closely connected with Gal$(\bar \Q/\Q)$},
  Leningrad Math. J. {\bf 2} (1991), 829--860.
  
\bibitem{GKZ}
H.~Gangle, M.~Kaneko, D.~Zagier, Double zeta values and modular forms, in `{\em Automorphic forms and Zeta functions}', 
Proceedings of the conference in memory of Tsuneo Arakawa, World Scientific, (2006), 71--106.

\bibitem{Gla}
C.~Glanois, Unramified Euler sums and Hoffman $\star$ basis, preprint,  arXiv:\,1603.05178.

\bibitem{Hoffman16}  M.~Hoffman, 
An odd variant of multiple zeta values, preprint, arXiv:\,1612.05232.


\bibitem{IKZ}  K.~Ihara, M.~Kaneko and D.~Zagier, 
Derivation and double shuffle relations for multiple zeta values, Compositio Math., {\bf 142} (2006), 307--338.

\bibitem{IKT2014}  K.~Imatomi, M.~Kaneko and E.~Takeda, 
Multi-poly-Bernoulli numbers and finite multiple zeta values, 
J. Integer Sequences, {\bf 17} (2014), Article 14.4.5.

\bibitem{Kaneko1997}  M.~Kaneko, 
Poly-Bernoulli numbers, J. Th\'eor. Nombres Bordeaux, {\bf 9} (1997), 199--206.


%\bibitem{KMT}  M.~Kaneko, P.~Maneka and H.~Tsumura, 
%On poly-cosecant numbers, in preparation.


\bibitem{KS}  M.~Kaneko and M.~Sakata,  
On multiple zeta values of extremal height, Bull. Aust. Math. Soc., {\bf 93} (2016), 186--193.

\bibitem{KanekoTasaka}
M.~Kaneko and K.~Tasaka, Double zeta values, double Eisenstein series,
and modular forms of level $2$, Math. Ann. {\bf 357} (2013), 1091--1118.

\bibitem{KT}  M.~Kaneko and H.~Tsumura,  
Multi-poly-Bernoulli numbers and related zeta functions, Nagoya Math. J., {\bf 232} (2018), 19--54.

\bibitem{KT2}  M.~Kaneko and H.~Tsumura,  
Zeta functions connecting multiple zeta values and poly-Bernoulli numbers, preprint, to appear in Adv. Stud. Pure Math. (arXiv:1811.07736).



\bibitem{M} T.~Machide, 
Extended double shuffle relations and generating function of triple zeta values of any fixed weight, Kyushu J. Math. {\bf 67} (2013), 281--307.


\bibitem{Ohno-Zagier} 
Y.~Ohno and D.~Zagier, Multiple zeta values of fixed weight, depth, and height, Indag. Math.,{\bf 12} (2001), 483-487.

\bibitem{OZ} 
Y.~Ohno and W.~Zudilin, Zeta stars, Commun. Number Theory Phys. {\bf 2} (2008), 325--347.

\bibitem{Sasaki2012}  Y.~Sasaki, 
On generalized poly-Bernoulli numbers and related $L$-functions,
J. Number Theory, {\bf 132} (2012), 156--170.


\bibitem{Tsu2004}
H.~Tsumura, Combinatorial relations for Euler-Zagier sums, Acta Arith. {\bf 111 } (2004), 27--42. 

\bibitem{Tsu2007}
H.~Tsumura, On the parity conjecture for multiple $L$-values of conductor four, Tokyo J. Math.  {\bf 30}  (2007), 21--40.

\bibitem{Tsu-parity}
H.~Tsumura, A note on the parity result for multiple $T$-values, in preparation.

\bibitem{Titch} 
E.~C.~Titchmarsh, {The theory of the Riemann zeta-function, Second edition}, edited and with a preface by D. R. Heath-Brown, The Clarendon Press, Oxford University Press, New York, 1986.

\if0
\bibitem{Yamamoto-B} 
S.~Yamamoto, Multiple zeta-star values and multiple integrals,
RIMS Kokyuroku Bessatsu B68, 2017, pp. 3--14.

\bibitem{Yamamoto} 
S.~Yamamoto, Multiple zeta functions of Kaneko-Tsumura type and their values at positive integers, preprint, arXiv:~1607.01978.
\fi

\bibitem{Z} 
D.~Zagier, Values of zeta functions and their applications, in ECM volume, Progress in Math., {\bf 120} (1994), 497--512.



\bibitem{Zh}
J.~Zhao, {Multiple zeta functions, multiple polylogarithms and their special values},
Series on Number Theory and its Applications, {\bf 12}, World Scientific Publishing Co. Pte. Ltd., Hackensack, NJ, 2016.


\end{thebibliography}
\end{document}